\documentclass{amsproc}
\usepackage{tikz}
\usepackage{amsmath}
\usepackage{amssymb}
\usepackage{mathtools}
\newtheorem{theorem}{Theorem}[section]
\newtheorem{lemma}[theorem]{Lemma}
\newtheorem{corollary}[theorem]{Corollary}

\theoremstyle{definition}

\newtheorem{example}[theorem]{Example}

\theoremstyle{remark}
\newtheorem{remark}[theorem]{Remark}

\numberwithin{equation}{section}

\begin{document}
	
	\title{The $n$-total graph of an integral domain}
	
	
	\author{Myriam AbiHabib}
	\address{Department of Mathematics and Statistics, American University of Sharjah, P.O. Box 26666, Sharjah, United Arab Emirates}
	\email{g00096222@alumni.aus.edu}
	
	
	\author{Ayman Badawi}
	\address{Department of Mathematics and Statistics, American University of Sharjah, P.O. Box 26666, Sharjah, United Arab Emirates}
	
	\email{abadawi@aus.edu}
	
	\vskip0.2in
	
	\begin{abstract}
	Let $R$ be a finite product of integral domains and $D$ be a union of prime ideals (it is possible that $R$ is just an integral domain). Let $n \geq 1$ be a positive integer. This paper introduces the $n$-total graph of a $(R, D)$. The $n$-total graph of  $(R, D)$, denoted by $n-T(R)$, is an undirected simple graph with vertex set $R$, such that two vertices $x, y$ in $R$ are connected by an edge if $x^n + y^n \in D$. In this paper, we study some graph properties and theoretical ring structure.

	\end{abstract}
	
	\maketitle
	\centerline{Dedicated to Professors  Mohammad Ashraf and  R. K. Sharma }
\section{Intrudection}

Throughout this paper, all rings are commutative with $1\not = 0$. 
The use of graph theory to study algebraic structures has become increasingly popular since Beck's work in \cite{9}, in which he presented a coloring of commutative rings based on multiplication to zero. The result of this was the formalization of the zero-divisor graph, a construction where distinct vertices $x$ and $y$ are adjacent if $xy=0$. Additional research by Anderson and Naseer \cite{3} and Anderson and Livingston \cite{8} established the fundamental characteristics of zero-divisor graphs, including their connectivity and diameter bounds. In \cite{5}, Anderson and Badawi introduced the total graph, where adjacency is based on the condition $x+y \in Z(R)$, to investigate additive behavior within rings.

The concept was generalized further in \cite{7}, where Anderson and Badawi defined the generalized total graph $GT(R, H)$, using a subset $H \subseteq R$, such that two distinct vertices $x, y \in R$ are adjacent (i.e., connected by an edge) if $x+y \in H$. A consequence of the study in \cite{7}, the total graph of a commutative ring is shown to be a special case, allowing the study of the total graph of integral domains.

Related advancements increased the number of tools available for connecting graphs to rings. Akbari $\&$ Co. Regular graphs of commutative rings were examined by \cite{2}, who compared their structure to that of total graphs and looked into adjacency defined by multiplicative and additive conditions. The subsequent work by Anderson and Badawi in \cite{6} altered the total graph by eliminating the zero element, resulting in a more refined graph structure that helps identify particular ring-theoretic properties, especially in reduced rings.

A unified treatment of these constructions is offered by foundational overviews such as Graphs from Rings \cite{4}, which show how different algebraic behaviors, including multiplicative, additive, and module-theoretic ones, can be modeled by graphs. A long list of researchers who have contributed significantly to the theory of graphs from an algebraic perspective can be found in \cite{4}.

Let $G(V, E)$ be a graph. We recall that $G = (V,E)$ is called  {\it connected} if there is a path from $u$ to $v$ in $G$ for all vertices $u, v \in V$; $G(V, E)$ is said to be {\it disconnected} if there is a pair of vertices $u, v$ in $V$ that do not have a path between them; and 
$G = (V,E)$ is called {\it totally disconnected} if there is
no path from $u$ to $v$ in $G$ for all vertices $u, v \in V$. The {\it distance} between two vertices $u$ and $v$ in a graph $G = (V,E)$ is
the length of the shortest path between $u$ and $v$ in $G$. It is denoted by $d(u, v)$. If there is no path, we say $d(u, v) = \infty$. The diameter of a graph $G = (V,E)$ is the supremum of distances between vertices $u$ and $v$ in $V$. We denote it by $diam(G)$.The girth of a graph $G = (V,E)$ is the length of the shortest cycle
in $G$ denoted by $gr(G)$. If there are no cycles in $G$, we say $gr(G) = \infty$. A {\it complete} graph $G = (V,E)$ is a graph in which every vertex
$v \in V$ is adjacent to every other vertex $u \in V$ . A complete graph with $n$ vertices is denoted $K_n$. A {\it complete bipartite} graph $G = (V,E)$ is a graph in which the set of vertices is partitioned into two sets $V_1, V_2$ such that $V1 \cap V2 = \emptyset$, every two vertices in $V_1$ are not adjacent, and every two vertices in $V_2$ are not adjacent. In addition, every vertex in $V_1$ is adjacent to every vertex in $V_2$. A {\it complete bipartite} graph with $|V_1| = m$ and $|V_2| = n$ is denoted $K_{n,m}$. A subgroup $C$ of a graph $G$ is called a {\it component} of $G$ if $C$ is connected and every vertex in $C$ is not connected to every vertex outside $C$. 

Recently, AitElhadi and Badawi \cite{1} generalized the concept of the total graph of a commutative ring to $n$-total graph of a commutative ring. Recall from \cite{1}, let $R$ be a commutative ring with $Z(R)$ as the set of all zero divisors of $R$, and $n \geq 1$. {\it The $n$-total graph of $R$}, denoted by $n-T(R)$ is a simple undirected graph with vertex set $R$ such that two vertices $x,y$ in $R$ are adjacent if $x^n+y^n \in Z(R)$. Since $Z(R)$ is a union of prime ideals, in this study we replace    $Z(R)$ in \cite{1} by a set $D$ that is a union of prime ideals. We define a new graph over $(R, D)$, where $R$ is a finite product of integral domains and $D$ is a union of incomparable (under inclusion) prime ideals of $R$. Notice that $R$ may be an integral domain. Let $n \geq 1$ be a positive integer.{\it  The $n$-total graph} of  $(R, D)$, denoted by $n-T(R)$, is an undirected simple graph with vertex set $R$, such that two vertices $x, y$ in $R$ are connected by an edge (i.e., adjacent) if $x^n + y^n \in D$.

In this paper, we analyze the properties of the n-total graph of an integral domain with respect to a union of prime ideals, focusing on connectedness, diameter, and girth.
Notice that the subgraph $n-TG(D)$ of the $n-TG(R)$ is always connected, and the $n-TG(D)$ is complete if and only if $D$ is an ideal of $R$. Moreover, if $D$ is an ideal of $R$, then the induced subgraphs $n-TG(D)$ and $n-TG(R\setminus D)$ are disjoint subgraphs of the $n-TG(R)$. However, if $D$ is not an ideal of R, then the induced subgraphs $n-TG(D)$ and $n-TG(R\setminus D)$ of the $n-TG(R)$ are never disjoint.

\section{The case where $D$ is an ideal}
Let $n \geq 1$ be an integer and $R$ be a finite field such that $D$ is a prime ideal of $R$. This implies that $D = \{0\}$ and $|D|=1$.
Let $|R|=m$ and $d=gcd(n,m-1)$. It is clear that the induced subgraph $n-TG(D)$ is a complete component of the $n-TG(R)$. Therefore, this section will focus on the $n-TG(R \setminus D)$. Let $n \geq 1$ and $S_n = \{a^n  |   a \in  R \setminus D\}$. The following known Lemmas are used as tools in order to prove some of the results in this paper. We leave the proofs to the reader. 

\begin{lemma}\label{lem:1}
	Let $R$ be a finite field with $m$ elements and $d=gcd(n,m-1)$.
	Then   $S_n = \{a^n  |   a \in  R \setminus D\}$  is a cyclic subgroup of $R \setminus D$ and $|S_n| = \frac{m-1}{d}$.
\end{lemma}

\begin{lemma}\label{lem:2}
	Let $R$ be a finite field. Then $x^n = a$  has a solution for some $a \in R\setminus D$ if and only if $a \in S_n$. Furthermore,  if $a \in S_n$, then  $x^n = a$  has exactly $d$ solutions.
\end{lemma}
 We invite the reader to compare the following result with \cite[ Theorem 2.6]{1}.
\begin{theorem}
	Let $R$ be a finite field with $m$ elements such that $char(R) = 2$ (i.e., $1 + 1 = 0$), $n \geq 1$ be an integer,  $d = gcd(n, m-1)$, and $\alpha = \frac{m-1}{d}$.
	Then, $n-TG(R) = K_d \oplus K_d \oplus …\oplus (\alpha\  times) \oplus  K_1$.
\end{theorem}

\begin{proof}
	For every $a \in S_n$, let $C_n(a) = \{b \in R\setminus D\  | b^n = a\}$.
	Then  $|C_n(a)| =  d$  by Lemma~\ref{lem:2}. Since $char(R ) = 2$, every two distinct vertices in $C_n(a)$ are adjacent in the $n-TG(R)$. Hence, $C_n(a)$ elements form the complete component $K_d$.  Since $|S_n| = \alpha $ by  Lemma~\ref{lem:1},  we have exactly $\alpha$ complete components $K_d$.  Since D forms the component $K_1$  of the $n-TG(R)$ , we have $n-TG(R) = K_d \oplus K_d \oplus … \oplus K_d  \ (\alpha \  times) \oplus  K_1$.
\end{proof}

\begin{example}
	Let $\mathbb{F}_4=\mathbb{F}_2[X]/(f(x))$, where $f(x)=x^2+x+1$. Then, $3-TG(\mathbb{F}_4)=K_3 \oplus K_1$ as shown in Figure \ref{fig:F4graph}.
\end{example}

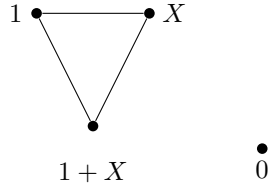
\begin{figure}[h]
	\centering
	\begin{tikzpicture}[scale=1.5, every node/.style={circle, inner sep=1pt}]
		\node[fill, circle, minimum size=4pt, inner sep=0pt, label=left:{$1$}] (v1) at (0,1) {};
		\node[fill, circle, minimum size=4pt, inner sep=0pt, label=right:{$X$}] (v2) at (1,1) {};
		\node[fill, circle, minimum size=4pt, inner sep=0pt, label=below:{$1+X$}] (v3) at (0.5,0) {};
		\node[fill, circle, minimum size=4pt, inner sep=0pt, label=below:{$0$}] (v0) at (2,-0.2) {};
		
		\draw (v1) -- (v2);
		\draw (v1) -- (v3);
		\draw (v2) -- (v3);
		
		
	\end{tikzpicture}
	\caption{The \(3\text{-TG}(\mathbb{F}_4)\)}
	\label{fig:F4graph}
\end{figure}

\begin{example}
	Let $\mathbb{F}_4=\mathbb{F}_2/(f(x))$ where $f(x)=x^2+x+1$ and let $n=5$.
	Here $d=1$ and $\alpha=3$. Then, the $5-TG(\mathbb{F}_4 \setminus D)$ is totally disconnected.
\end{example}
 We invite the reader to compare the following result with \cite[ Theorem 2.9]{1}.
\begin{theorem}\label{T2}
	Let $R$ be a finite field of order $m$ such that $char(R) \neq 2$ ,  $n\geq 1$ be an integer,  $d = gcd(n, m-1)$, and $\alpha = \frac{m-1}{d}$. Then
	\begin{enumerate}
		\item  The $n-TG(R \setminus D)$ is totally disconnected if and only if $\alpha$ is odd.
		\item If $\alpha$ is even, then the $n-TG(R \setminus D)=K_{d,d} \oplus ... \oplus K_{d,d}$ , $\frac{\alpha}{2}$ times.
	\end{enumerate}
\end{theorem}

\begin{proof}
	\begin{enumerate}
		\item Suppose that the $n-TG(R \setminus D)$ is totally disconnected. Assume that $ \alpha$ is even. Then $S_n(R) = \{b^n \mid b \in R^*\}$ is the unique cyclic subgroup of $R^*$ of order $\alpha$ by Lemma~\ref{lem:1}. Since $\alpha$ is even, $-1  \in C_n(R)$ and since $char(R) \not = 2,$ we have $1  \not = -1 $.  Hence, there are $x, y \in R \setminus D$ such that $x^n + y^n = 0 \in R$; a contradiction since the $n-TG(R \setminus D)$ is totally disconnected.\\
		For the converse, suppose that $\alpha$ is odd and the $n-TG(R \setminus D)$ is not totally disconnected. Then, there exists $x,y \in D$ such that $x^n+y^n \in D$. Thus, $x^n=-y^n$. Raising both sides to the power of $\alpha$, we get $(x^n)^\alpha = (-y^n)^\alpha$, which implies that $x^{n\alpha}=-y^{n\alpha}$ since $\alpha$ is odd. Hence, $1=-1$. This is a contradiction, since $char(R) \neq 2$. Therefore, the $n-TG(R\setminus D)$ is totally disconnected.
		\item Suppose that $x$ and $y$ are adjacent in the $n-TG(R\setminus D)$. Then
		$ x^n+y^n \in D$. This implies that $x^n=-y^n$.
		Let $a=x^n$ and $b=-y^n$, then each of these equations will have $d$ solutions by Lemma 2. Hence, we will have $d$ vertices connected to $d$ other vertices. Now, suppose $x_1$ and $x_2$ are two solutions of the equation $a=x^n$. Then $x_1^n+x_2^n=a+a=2a \notin D$. Similarly, if $y_1$ and $y_2$ are solutions of the equation $b=-y^n$, then they are not adjacent. Therefore, we have a complete bipartite graph $K_{d,d}$. Since each complete bipartite has $2d$ vertices and $|R\setminus D|=m-1$, then the $n-TG(R \setminus D)$ will consist of $\frac{m-1}{2d}=\frac{\alpha}{2}$ complete bipartites $K_{d,d}$.
	\end{enumerate}
\end{proof}

\begin{example}
	The $n-TG(\mathbb{Z}_m \setminus D)$ is connected where m is an odd prime and $D=\{0\}$.
	Take $m=7$ and $n=3$.
	Then, $d=3$ and $\alpha=2$. Hence, $3-TG(\mathbb{Z}_7 \setminus D)=K_{3,3}$ as shown in Figure \ref{fig:F5graph}.
\end{example}

\begin{figure}[!h]
	\centering
	\begin{tikzpicture}
		\tikzset{enclosed/.style={draw, circle, inner sep=0pt, minimum size=.15cm, fill=black}}
		
		\node[enclosed, label={below: 1}] (1) at (0,0) {};
		\node[enclosed, label={below: 2}] (2) at (1.5,0) {};
		\node[enclosed, label={below: 4}] (4) at (3,0) {}; 
		
		\node[enclosed, label={above: 3}] (3) at (0,2) {}; 
		\node[enclosed, label={above: 5}] (5) at (1.5,2) {};
		\node[enclosed, label={above: 6}] (6) at (3,2) {};
		
		\foreach \a in {1,2,4}
		\foreach \b in {3,5,6}
		\draw (\a) -- (\b);
	\end{tikzpicture}
	\caption{The $3\text{-}TG(\mathbb{Z}_7 \setminus D)$}
	\label{fig:F5graph}
\end{figure}
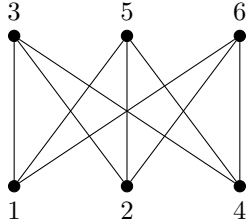

\begin{example}
	The finite field $\mathbb{F}_9$.
	Let $f(x)=x^2+1$. Then, the $\mathbb{F}_9=\mathbb{F}_3/(f(x))$ and $n=5$.
	Here $d=1$ and $\alpha=8$. Then, $5-TG(\mathbb{F}_9 \setminus D)=K_{1,1} \oplus ... \oplus K_{1,1}$ , $4$ times as shown in Figure \ref{fig:F6graph}.
\end{example}

\begin{figure}[!h]
	\centering
	\begin{tikzpicture}[xscale=2.5]
		\tikzset{enclosed/.style={draw, circle, inner sep=0pt, minimum size=.15cm, fill=black}}
		
		\node[enclosed, label={below: $1$}] (a1) at (0,0) {};
		\node[enclosed, label={above: $x$}] (b1) at (0,1) {};
		\draw (a1) -- (b1);
		
		\node[enclosed, label={below: $2$}] (a2) at (1,0) {};
		\node[enclosed, label={above: $2x$}] (b2) at (1,1) {};
		\draw (a2) -- (b2);
		
		\node[enclosed, label={below: $1+x$}] (a3) at (2,0) {};
		\node[enclosed, label={above: $2+2x$}] (b3) at (2,1) {};
		\draw (a3) -- (b3);
		
		\node[enclosed, label={below: $1+2x$}] (a4) at (3,0) {};
		\node[enclosed, label={above: $2+x$}] (b4) at (3,1) {};
		\draw (a4) -- (b4);
	\end{tikzpicture}
	\caption{The $5$-TG$(\mathbb{F}_9 \setminus D)$}
	\label{fig:F6graph}
\end{figure}
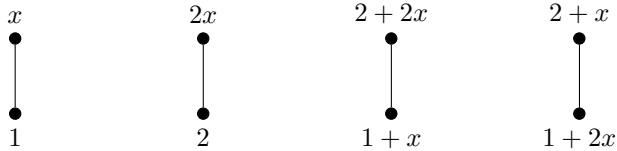

 We invite the reader to compare the following result with \cite[ Corollary 2.12]{1}.

\begin{corollary}
	Let $n \geq 1$ be an integer and R a finite field of order $m$ such that $char(R) \neq 2$. Then, the $n-TG(R \setminus \{0\})$ is connected if and only if $d= \frac{m-1}{2}$.
\end{corollary}

\begin{proof}
	Suppose that $d= \frac{m-1}{2}$. Since $\alpha = \frac{m-1}{d}$, then $\alpha = 2$. By Theorem~\ref{T2}, the $n-TG(R \setminus D)=K_{d,d} $ is connected. Now, suppose that the $n-TG(R \setminus D)=K_{d,d} $ is connected. Then, the $n-TG(R)$ has $d+d+1$ vertices which mean that $m=2d+1$, and hence $d=\frac{m-1}{2}$.
\end{proof}
 We invite the reader to compare the following result with \cite[ Theorem 2.15]{1}.
\begin{theorem}
	Let $R$ be a finite field and $D=\{0\}$. Then
	\begin{enumerate}
		\item $diam(n-TG(R \setminus D))=$  $0$ , $1 $, $2$, or $\infty$
		\item $gr(n-TG(R \setminus D))= 3,4$ or $\infty$.
	\end{enumerate}
\end{theorem}

\begin{proof}
	\begin{enumerate}
		\item The $n-TG(R \setminus D))$ is either totally disconnected, the union of complete subgraph of $d$ vertices, or the union of complete bipartite. Hence, the diameter is either $0$,$1$ or $2$.
		\item Suppose that the $n-TG(R \setminus D)$ has a cycle. Then, it is either the union of complete subgraph, or the union of complete bipartite. Hence, the girth is either $3$ or $4$.
		If the $n-TG(R \setminus D)$ is totally disconnected, then $gr(n-TG(R \setminus D))= \infty$.
	\end{enumerate}
\end{proof}

\begin{example}
	\begin{enumerate}
		\item $diam(3-TG(\mathbb{Z}_7 \setminus D))=2$ and $gr(3-TG(\mathbb{Z}_7 \setminus D))=4$.
		\item $diam(5-TG(\mathbb{F}_9 \setminus D))=1$ and $gr(5-TG(\mathbb{F}_9 \setminus D))=\infty$ since the graph has no cycle.
		\item $diam(5-TG(\mathbb{F}_4 \setminus D))=\infty$ and $gr(5-TG(\mathbb{F}_4 \setminus D))=\infty$ since the graph is totally disconnected.
	\end{enumerate}
\end{example}

\begin{theorem}
	Let $R=R_1\times...\times R_k$, where $R_i$ is an integral domain for $ 1 \leq i \leq k$, where $k \geq 2$, and let $D = P_1 \times \prod_{i=2}^k R_i $. If $n \geq 1$, then the $n-TG(R)$ is never connected.
\end{theorem}

\begin{proof}
	Suppose that the $n-TG(R)$ is connected. Then, there exists a path from $(0,0,...0)$ to $(1,1,...,1)$. Hence, the $n-TG(R_1)$ is connected. Suppose that $0-w_1-w_2-...-w_m-1_{R_1}$ is a path in $R_1$. Then, $w_1^n \in P_1, w_1^n+w_2^n \in P_1,..., w_m^n+1 \in P_1.$ This implies that $1 \in P_1.$ Which is a contradiction.
\end{proof}

\subsection{The case where $D$ is not an ideal of $R$}

In this section, we will consider $R$ to be an integral domain or a product of integral domains, and $D$ is a union of at least two incomparable prime ideals of $R.$

\begin{remark}
	\begin{enumerate}
		Let $R$ be an integral domain and $D=P_1\cup P_2 \ \cup ... \cup P_k \ , \ k \geq 2$ such that for every $1 \leq i,j \leq k, P_i \nsubseteq P_j$.
		\item The $n-TG(D)$ is always connected since for all $x,y \in D$, $x-0-y$ is a path.
		\item The $n-TG(R)$ is connected if and only if $1=x_1+x_2+...+x_m$ where $x_i \in D$ for $1 \leq i \leq m$.
	\end{enumerate}
\end{remark}
 We invite the reader to compare the following result with \cite[ Theorem 2.17]{1}.
\begin{theorem}\label{T3}
	Suppose that $R$ is an integral domain and $D=P_1 \cup P_2 \cup ... \cup P_k $, where $k \geq 2$ such that for every $1 \leq i,j \leq k, P_i \nsubseteq P_j$. Then, the $n-TG(R)$ is connected if and only if there exists a path between $0$ and $1$.
\end{theorem}

\begin{proof}
	Suppose that the $n-TG(R)$ is connected, then there exists a path between any two vertices.\\
	Now suppose that there exists a path between $0$ and $1$. Then, $0 - v_1 - v_2 - ... - v_k - 1$. Hence, for any $v_i - v_j$, $v_i^n+v_j^n \in D$.
	Let $w \in R$, then $(wv_i)^n+(wv_j)^n=w^n(v_i+v_j) \in D$. Therefore, $0 - wv_1 -w v_2 - ... - wv_k - w$ is a path from $0$ to $w$, for all $w \in R$ and hence there exists a path between any two vertices in $R$. Therefore, the $n-TG(R)$ is connected.
\end{proof}

\begin{theorem} \label {coprime}
	Suppose that $R$ is an integral domain. Let $n \geq 1$ be an odd integer and $D=P_1\cup P_2 \ \cup ... \cup P_k \ , \ k \geq 2$ such that for every $1 \leq i,j \leq k, P_i \nsubseteq P_j$. Suppose that at least two of the $P_k$'s are coprime, then the $n-TG(R) $ is connected with diameter $2$.
\end{theorem}

\begin{proof}
	We may assume without loss of generality that $P_1$ and $ P_2$ are coprime. Then, $P_1+P_2=R$ and there exist $p_1 \in  P_1, p_2 \in P_2$ such that $p_1+p_2=1$. Let $x,y \in R$ such that $x=xp_1+xp_2$ and $y=yp_1+yp_2$. Assume that $x$ is not adjacent to $y$ and let $w=-xp_1-yp_2$. Then, $x^n=x^n p_1^n +x^np_2^n +p_1p_2k , \ k \in R$ and $w^n=-x^np_1^n-y^np_2^n+p_1p_2k_1,\ k_1 \in R$. Hence, $x^n+w^n=p^n(x^n-y^n)+p_1p_2k+p_1p_2k_1 \ \in P_2$. Therefore, $x^n+w^n \in D$. Similarly,
	$w^n+y^n=-x^np_1^n+p_1p_2k_1+y^np_1^n+p_1p_2k_3,\ k_3 \in R$, then $w^n+y^n \in D.$ This implies that $w$ is adjacent to $x$ and $y$ and hence $diam(n-TG(R))=2$.
\end{proof}

\begin{corollary}
	Suppose that $R$ is an integral domain. Let $n \geq 1$ be an odd integer and  $D=M_1\cup M_2 \ \cup ... \cup M_k \ , \ k \geq 2$, where $M_i$ is a maximal ideal for all $1 \leq i \leq k$. Then the $n-TG(R) $ is connected with diameter $2$.
\end{corollary}

\begin{theorem}
	Suppose that $R$ is an integral domain. Let $n \geq 1$ be an even integer and $D=P_1\cup P_2 \ \cup ... \cup P_k \ , \ k \geq 2$  such that for every $1 \leq i,j \leq k, P_i \nsubseteq P_j$. Suppose that at least two of the $P_k$'s are coprime and there exists a $u\in R$ such that $u^n=-1$. Then the $n-TG(R) $ is connected with diameter 2.
\end{theorem}

\begin{proof}
	We may assume without loss of generality that $P_1$ and $ P_2$ are coprime. Then $P_1+P_2=R$ and there exist $p_1 \in  P_1, p_2 \in P_2$ such that $p_1+p_2=1$. Let $x,y \in R$ such that $x=xp_1+xp_2$ and $ y=yp_1+yp_2$. Assume that $x$ is not adjacent to $y$ and let $w=uxp_1+uyp_2$. We have
	$x^n=x^n p_1^n +x^np_2^n +p_1p_2k , \ k \in R$ and
	$w^n=-x^np_1^n-y^np_2^n+p_1p_2k_1,\ k_1 \in R$.
	Then $x^n+w^n=p^n(x^n-y^n)+p_1p_2k+p_1p_2k_1 \ \in P_2$. Therefore, $x^n+w^n \in D$. Similarly, $w^n+y^n=-x^np_1^n+p_1p_2k_1+y^np_1^n+p_1p_2k_3,\ k_3 \in R$, then $w^n+y^n \in D.$ This implies that $w$ is adjacent to $x$ and $y$ and hence $diam(n-TG(R))=2$.
\end{proof}

\begin{example}
	Let $n \geq 1$ be an odd integer and $D=p_1 \mathbb{Z} \cup p_2 \mathbb{Z}$ where $p_1$ and $p_2$ are two distinct primes. Then, the $n-TG( \mathbb{Z} )$ is connected and $diam(n-TG(\mathbb{Z}))=2$. In fact, $gcd(p_1,p_2)=1$. Thus, $1=c_1p_1+c_2p_2$ for some $c_1,c_2 \in \mathbb{Z}$. Hence, $c_1p_1-1=-c_2p_2 \in p_2 \mathbb{Z}$. Thus, $0^n+(c_1p_1-1)^n \in p_2\mathbb{Z}$ and $0^n+(-c_2p_2)^n \in p_2 \mathbb{Z}$ which means that there exists a path from 0 to 1.
\end{example}

In the following example, we show that the hypothesis, at least two prime ideals are coprime in Theorem \ref{coprime}, is crucial.

\begin{example}
	Let $R=\mathbb{Z}[X,Y]$ and $D=XR \cup YR$. Since  $1  \not =  a + b$ for some $a \in XR$ and $b \in YR$, there is no path from $0$ to $1$.  Hence, for every integer $n \geq 1$, the $n-TG(R)$ is not connected.
\end{example}
 We invite the reader to compare the following result with \cite[ Theorem 2.18]{1}.
\begin{theorem}
	Suppose that $R$ is an integral domain and $D= P_1 \cup ...\cup P_k $ where $k \geq 2$ such that for all $1 \leq i,j \leq k , P_i \not \subseteq P_j$. Let $n \geq 1$ be an odd integer. If the $n-TG(R \setminus D)$ is connected, then the $n-TG(R)$ is connected.
\end{theorem}

\begin{proof}
	Suppose that the $n-TG(R \setminus D)$ is connected. We know that the $n-TG(D)$ is connected. Now, since $D$ is not an ideal of $R$, there exist $x,y \in D \setminus \{0\}$ such that $x-y=r \in R \setminus D$.
	Therefore, $(x-y)^n=r^n = \sum^{n}_{k=1}  {n \choose
		k} (-1)^k x^k y^{n-k}= \sum^{n-1}_{k=1}  {n \choose
		k}(-1)^k x^k y^{n-k}-y^n=xq-y^n$ for some $q\in R$. This implies that $r^n+y^n=xq\in D$; hence, the $n-TG(R)$ is connected.
\end{proof}
 We invite the reader to compare the following result with \cite[ Theorem 2.20].

\begin{theorem}
	Suppose that $R$ is an integral domain and $D= P_1 \cup ...\cup P_k $ where $k \geq 2$ such that for all $1 \leq i,j \leq k , P_i \not \subseteq P_j$. Let $n \geq 1$ be an even integer and suppose that there exists $u \in R$ such that $u^n=-1$. If the $n-TG(R \setminus D)$ is connected then the $n-TG(R)$ is connected.
\end{theorem}

\begin{proof}
	Suppose that the $n-TG(R \setminus D)$ is connected. We know that the $n_TG(D)$ is connected. Now, since $D$ is not an ideal of $R$, there exist $x,y \in D \setminus \{0\}$ such that $x+uy=r \in R \setminus D$.
	Therefore $(x+uy)^n=r^n = \sum^{n}_{k=1}  {n \choose
		k} (-1)^k u^k x^k y^{n-k}= \sum^{n-1}_{k=1}  {n \choose
		k}(-1)^k u^k x^k y^{n-k}-y^n=xq-y^n$ for some $q\in R$. This implies that $r^n+y^n=xq\in D$; hence, the $n-TG(R)$ is connected.
\end{proof}
We invite the reader to compare the following result with \cite[Theorem 2.22]{1}.
\begin{theorem}
	Suppose that $R$ is an integral domain and $D= P_1 \cup \cdots\cup P_k $ where $k \geq 2$ such that for all $1 \leq i,j \leq k , P_i \not \subseteq P_j$. Let $n \geq 1$ be an odd integer. The $n-TG(R)$ is connected if and only if $R=(x_1,x_2,...,x_m)$, where $x_r\in D$ for each $1\leq r\leq m$.
\end{theorem}

\begin{proof}
	Suppose that the $n-TG(R)$ is connected. Then there is a path $0- v_1- v_2- ...- v_q- 1$ in the $n-TG(R)$. This implies that $v_1^n, v_1^n+v_2^n,..., v_{q-1}^n + v_q^n, v_q^n+1 \in D$. Hence, $1\in (v_1,v_1^n+v_2^n,..., v_{q-1}^n + v_q^n, v_q^n+1)\subseteq (x_1,...,x_m) $. Hence, $R=(x_1,...,x_m)$
	
	Conversely, suppose that $R=(x_1,x_2,...,x_m)$ where $x_r\in D$ for each $1\leq r\leq m$. Let $y \in R$ such that $y \neq 0 $. Then, we have $y=v_1+...+v_q$ for some $v_1,...,v_q\in D$. Let $w_0=0$ and $w_j=(-1)^{q+j}(v_1+...+v_j)$ for each integer $1\leq j\leq q$. Since n is odd, we have $(-1)^{n(q+j)}=(-1)^{q+j}$ and hence we have $w_j^n=(-1)^{q+j}(v_1+...+v_j)^n = w_j^n+w_{j+1}^n=v_{j+1}z \in D$ for some $z\in R$. Thus, $0- w_1- w_2-...- w_q-y$ is a path from $0$ to $y$ in the $n-TG(R)$. Let $0\neq a,b\in R$. Then, there are paths from $a$ to $0$ and from $0$ to $b$. Therefore, there is a path from $a$ to $b$ in the $n-TG(R)$; hence, the $n-TG(R)$ is connected.
\end{proof}
We invite the reader to compare the following result with \cite[Theorem 2.24]{1}.
\begin{theorem}
	Suppose that $R$ is an integral domain and $D= P_1 \cup ...\cup P_k $ where $k \geq 2$ such that for all $1 \leq i,j \leq k , P_i \not \subseteq P_j$. Let $n$ be an odd integer and suppose that $R=(x_1,x_2,...,x_m)$ where $x_r\in D$ for each $1\leq r\leq m$ (i.e. the $n-TG(R)$ is connected). Then \ $diam(n-TG(R))=m$, where $m$ is the minimum number of vertices $x_1,...,x_m$ in $D$ such that $(x_1,...x_m)=R$. Moreover, $diam(n-TG(R))=d(0,1)$.
\end{theorem}

\begin{proof}
	First, we show that any path from $0$ to $1$ in $R$ has a length of at least $m$. Suppose that $0- v_1- ...- v_{k-1}- 1$ is a path from $0$ to $1$ in the $n-TG(R)$ of length $k$. Then $v_1^n, v_1^n+v_2^n,..., v_{k-1}^n+1\in D$ and hence $1\in (v_1^n, v_1^n+v_2^n,..., v_{k-1}^n+1)\subseteq (x_1,...,x_m)$. Thus, $k\geq m$.
	Let $a$ and $b$ be distinct elements in $R$. We will show that there is a path in $R$ with a length less than or equal to $m$. Let $1=x_1+x_2+...+x_m$ for $x_1,..,x_m\in D$. If $m$ is even, define $z=a+b$. If $m$ is odd, define $z=a-b$. In any case, let $d_0=a$. For each $1\leq k\leq m$ let $d_k=-(a+z(x_1+...+x_k))$ if $k$ is odd, and $d_k=x+z(x_1+...+x_k)$ if $k$ is even. Then $d_k^n+d_{k+1}^n=x_{k+1}q \in D$ for some $q\in R$. Also, $d_m=b$. Therefore, $a- d_1-...- d_m-1- b$ is a path of length at most $m$. In particular, a shortest path between $0$ and $1$ would have length $m$. Therefore, $diam(n-TG(R))=m$.
\end{proof}
We invite the reader to compare the following result with \cite[ Theorem 2.25]{1}.
\begin{theorem}
	Let $n$ be an even integer, $R$ be an integral domain, and $D= P_1 \cup ...\cup P_k $, where $k \geq 2$ such that for all $1 \leq i,j \leq k , P_i \not \subseteq P_j$. Suppose that $R=(x_1,x_2,...,x_m)$ where $x_r\in D$ for each $1\leq r\leq m$.  If there is a $u\in R$ such that $u^n=-1$,  then the $n-TG(R)$ is connected and $diam(n-TG(R))= m$ where $m$ is the minimum number of vertices $x_1,...,x_m$ such that $(x_1,...,x_m)=R$. Moreover, $diam(n-TG(R))=d(0,1)$.
\end{theorem}

\begin{proof}
	First, we show that any path from $0$ to $1$ in $R$ has a length of at least $m$. Suppose that $0- v_1- ...- v_{k-1}- 1$ is a path from $0$ to $1$ in the $n-TG(R)$ of length $k$. Then $v_1^n, v_1^n+v_2^n,..., v_{k-1}^n+1\in D$ and hence $1\in (v_1^n, v_1^n+v_2^n,..., v_{k-1}^n+1)\subseteq (x_1,...,x_m)$. Thus, $k\geq m$.
	Let $a$ and $b$ be distinct elements in $R$. We will show that there is a path in $R$ with a length less than or equal to $m$. Let $1=x_1+x_2+...+x_m$ for $x_1,...x_m\in D$. If $m$ is even, define $z=a+b$. If $m$ is odd, define $z=a-b$. In any case, let $d_0=a$. For each $1\leq k\leq m$ let $d_k=u(a+z(x_1+...+x_k))$ if $k$ is odd and $d_k=a+z(x_1+...+x_k)$ if $k$ is even. Then $d_k^n+d_{k+1}^n=x_{k+1}q \in D$ for some $q\in R$. Also, $d_m=b$. Therefore, $a- d_1-...- d_{m-1}- b$ is a path of length at most $m$. In particular, a shortest path between $0$ and $1$ would have length $m$. Therefore, $diam(n-TG(R))=m$.
\end{proof}
We invite the reader to compare the following result with \cite[ Corollary 2.27]{1}.
\begin{corollary}
Let $R$ be a finite product of integral domains and $D$ be a union of incomparable prime ideals. If $diam(n-TG(R))=m$, then $diam(n-TG(R \setminus D)\geq m-2$.
\end{corollary}

\begin{proof}
	Since $diam(n-TG(R))=d(0,1)=m$, let $0- v_1- ...- v_{m-1}- 1$ be a path from $0$ to $1$ in $R$.  Clearly, $v_1\in D$. Suppose that for some $2\leq i\leq m-1$ we have $v_i$ in $D$. Therefore, we can have the path $0- v_i- ...- 1$ in the $n-TG(R)$ of length less than $m$. This is a contradiction. Thus, $v_i\in R \setminus D$ for each $2\leq i\leq m-1$ and hence $v_2-...- 1$ is a shortest path between $v_2$ and $1$ in the $n-TG(R \setminus D)$ of length $m-2$. Therefore $diam(n-TG(R\setminus D))\geq m-2$.
\end{proof}

\begin{example}
	Let $n \geq 1$ be an odd integer, $R=\mathbb{Z}[X]$ and $D=(2,X) \cup (3,X)$. We have $R=(2,3,X)$, then the $n-TG(R)$ is connected with $diam(n-TG(R))=3$ and $diam(n-TG(R\setminus D)) \geq 1$.
\end{example}

\begin{example}
	Let  $n \geq 1$ be an odd integer, $R=\mathbb{F_5}[X]$ and $D=(X) \cup (X-1)$. Since $1 = X - (X - 1)$, We have $R=(X,X-1)$. Hence the $n-TG(R)$ is connected with $diam(n-TG(R))=2$.
\end{example}

\begin{theorem}
	Let $R=R_1 \times R_2 \times ... \times R_k$ where $k \geq 2$ and $R_i$ an integral domain. Let $D=\bigcup_{j=1}^{\ell} \left( P_j \times \prod_{i=2}^k R_i \right)$ where $P_1,P_2,...,P_l$ are prime ideals of $R_1$. Then the $n-TG(R)$ is connected if and only if $P_r$ and $P_s$ are coprime in $R_1$ for $1 \leq r,s \leq l$.
\end{theorem}

\begin{proof}
	Suppose that the $n-TG(R)$ is connected. Then the $n-TG(R_1)$ is connected and Hence by Theorem \ref{coprime}, at least two prime ideals $P_r$ and $P_s$ are coprime in $R_1$ for $1 \leq r,s \leq l$. Now suppose that at least two prime ideals $P_r$ and $P_s$ are coprime in $R_1$ for $1 \leq r,s \leq l$. Then by Theorem \ref{coprime}, the $n-TG(R_1)$ is connected. Hence, there exists a path from $0$ to $1_{R_1}$. Hence the $n-TG(R)$ is connected.
\end{proof}

\begin{theorem}
	Let $R=R_1 \times R_2 \times ... \times R_k$ where $k \geq 2$ and $R_i$ an integral domain. Let $D = \cup Q_i$  such that each $Q_i$ is a prime ideal of $R$ . Suppose that for at least two prime ideals,  say $Q_i, Q_j, i \not = j$, we have $ Q_i = P_i \times \prod_{s \not = i} R_s$ and $Q_j = P_j \times  \prod_{s \not = j} R_s$, where $P_i$ is a prime ideal of $R_i$ and $P_j$ is a prime ideal of $R_j$. If $n \geq 1$ is an odd integer, then the $n-TG(R)$ is always connected with diameter 2.
\end{theorem}

\begin{proof}
	Let $x=(0,...,0)$ and $z=(1,...,1)$ $\in R$ and let $y=(0,...,p_i,...,-1+p_j,...,0)$ where $p_i \in P_i$ and $p_j \in P_j$ for $1 \leq i,j \leq k$. We have $0+p_i^n \in P_i$ which means that $x$ is adjacent to $y$. Furthermore, $(-1+p_j)^n+1=-1+p_j^n+cp_j+1 \in P_j$ for some $c \in R$. This implies that $z$ is adjacent to $y$ and the shortest path between $x$ and $z$ is 2.
\end{proof}

\begin{theorem}
	Let $R=R_1 \times R_2 \times ... \times R_k$ where $k \geq 2$ and $R_i$ an integral domain. Let $D = \cup Q_i$  such that each $Q_i$ is a prime ideal of $R$ . Suppose that for at least two prime ideals,  say $Q_i, Q_j, i \not = j$, we have $ Q_i = P_i \times \prod_{s \not = i} R_s$ and $Q_j = P_j \times  \prod_{s \not = j} R_s$, where $P_i$ is a prime ideal of $R_i$ and $P_j$ is a prime ideal of $R_j$. If $n \geq 1$ is an even integer and there exists an element $u$ in $R$ such that $u^n=-1$, then the $n-TG(R)$ is always connected with diameter 2.
\end{theorem}

\begin{proof}
	Let $x=(0,...,0)$ and $z=(1,...,1)$ $\in R$ and let $y=(0,...,p_i,...,u+p_j,...,0)$ where $p_i \in P_i$ and $p_j \in P_j$ for $1 \leq i,j \leq k$. We have $0+p_i^n \in P_i$ which means that $x$ is adjacent to $y$. Furthermore, $(u+p_j)^n+1=u^n+p_j^n+cup_j+1 \in P_j$ for some $c \in R$. This implies that $z$ is adjacent to $y$ and the shortest path between $x$ and $z$ is 2.
\end{proof}

\begin{example}
	Let $R=\mathbb{Z}_2 \times \mathbb{Z}_2$ and $D= \mathbb{Z}_2 \times \{0\} \cup \{0\} \times \mathbb{Z}_2$. The $3-TG(R)$ is connected with diameter 2 as shown in Figure~\ref{f8}.
\end{example}

\begin{example}
	Let $R=\mathbb{Z}_2 \times \mathbb{Z}_3$ and $D= \mathbb{Z}_2 \times \{0\} \cup \{0\} \times \mathbb{Z}_3$. Here we have $1^2=-1$ in $\mathbb{Z}_2$ and the $2-TG(R)$ is connected as shown in Figure~\ref{f9}.
\end{example}

\begin{figure}[h!]
	\centering
	\begin{tikzpicture}
		\tikzset{enclosed/.style={draw, circle, fill=black, inner sep=1pt, minimum size=6pt}}
		
		\node[enclosed, label=below:{$(0,0)$}] (A) at (0,0) {};
		\node[enclosed, label=below:{$(1,0)$}] (B) at (2,0) {};
		\node[enclosed, label=above:{$(1,1)$}] (C) at (2,2) {};
		\node[enclosed, label=above:{$(0,1)$}] (D) at (0,2) {};
		
		\draw (A) -- (B);
		\draw (B) -- (C);
		\draw (C) -- (D);
		\draw (D) -- (A);
	\end{tikzpicture}
	\caption{The $3$-TG$(\mathbb{Z}_2 \times \mathbb{Z}_2)$}
	\label{f8}
\end{figure}
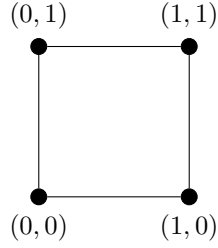

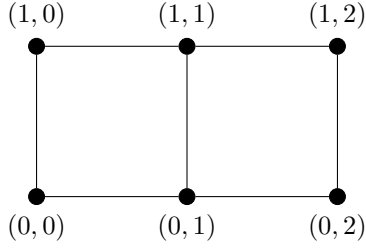
\begin{figure}[h!]
	\centering
	\begin{tikzpicture}
		\tikzset{enclosed/.style={draw, circle, fill=black, inner sep=1pt, minimum size=6pt}}
		
		\node[enclosed, label=below:{$(0,0)$}] (A) at (0,0) {};
		\node[enclosed, label=below:{$(0,1)$}] (B) at (2,0) {};
		\node[enclosed, label=below:{$(0,2)$}] (C) at (4,0) {};
		
		\node[enclosed, label=above:{$(1,0)$}] (D) at (0,2) {};
		\node[enclosed, label=above:{$(1,1)$}] (E) at (2,2) {};
		\node[enclosed, label=above:{$(1,2)$}] (F) at (4,2) {};
		
		\draw (A) -- (B);
		\draw (B) -- (C);
		\draw (A) -- (D);
		\draw (B) -- (E);
		\draw (C) -- (F);
		\draw (D) -- (E);
		\draw (E) -- (F);
	\end{tikzpicture}
	\caption{The $2$-TG$(\mathbb{Z}_2 \times \mathbb{Z}_3)$}
	\label{f9}
\end{figure}

\begin{example}
	Let $R=\mathbb{Z}_3 \times \mathbb{Z}_7=R_1\times R_2$ and $D= \mathbb{Z}_3 \times \{0\} \cup \{0\} \times \mathbb{Z}_7$. Here there does not exist an $x \in R_i$ such that $x^2=-1$, $i \in \{1,2\}.$ The $2-TG(R)$ is not connected.
\end{example}

\begin{example}
	Let $R= \mathbb{F}_9 \times \mathbb{F}_{25}$ and $D=\{0\} \times \mathbb{F}_{25} \cup \mathbb{F}_9 \times \{0\}= P_1 \cup P_2$. Since $P_1$ and $P_2$ are coprime, and there exists a $u$ in $R$ such that $u^2=-1$, then the $2-TG(R) $ is connected.
\end{example}

\begin{example}
	Let $R= \mathbb{F}_5 \times \mathbb{F}_5$ and $D=\{0\} \times \mathbb{F}_5 \cup \mathbb{F}_5 \times \{0\}= P_1 \cup P_2$. The $n-TG(R) $ is always connected.
\end{example}

\begin{example}
	Let $R=\mathbb{Z}_7 \times \mathbb{Z}$ and $D=\{0\} \times \mathbb{Z} \cup \mathbb{Z}_7 \times 3 \mathbb{Z}$. The $3-TG(R)$ is connected while the $2-TG(R)$ is not.
\end{example}

\begin{example}
	Let $R=\mathbb{F}_9 \times \mathbb{Z}[X]$ and $D=\{0\} \times \mathbb{Z}[X] \cup \mathbb{F}_9 \times P$ where $P$ is a prime ideal in $\mathbb{Z}[X]$. The $n-TG(R)$ is connected for $n \geq 1$ odd integer.
\end{example}

\begin{theorem}
	Let $R$ be an integral domain and $D= P_1 \cup ...\cup P_k $ where $k \geq 2$ such that for all $1 \leq i,j \leq k , P_i \not \subseteq P_j$. Let $n \geq 1$, we have the following
	\begin{enumerate}
		\item $gr(n-TG(D))=3$ or $gr(n-TG(D))=\infty$.
		\item If $n$ is odd, then the $gr(n-TG(R \setminus D))=3,4$ or $\infty$.
		\item If $n$ is even and $x^n=-1$ for some $x \in R$, then the $gr(n-TG(R \setminus D))=3,4$ or $\infty$.
	\end{enumerate}
\end{theorem}

\begin{proof}
	\begin{enumerate}
		\item Let $x,y \in D$. If $x^n+y^n \in D$, then $0-x-y-0$ is a cycle of length $3$ in the $n-TG(D)$.
		If for all $ x,y \in D\setminus \{0\}, x^n+y^n \notin D$, we will have a star graph since for all $ x \in D \setminus \{0\}$ we have $0^n+x^n \in D$. Therefore, the $gr(n-TG(D))=\infty$.
		\item Assume that $char(R) \neq 2$. Then $x \neq -x$ for each $x \in R$. Suppose that $R \setminus D$ contains a cycle. Then there exist $x, y \in R \setminus D$ such that $x \neq -y$ and $x^n + y^n \in D$. This implies that the sequence $x - y -( -y) -( -x) - x$ forms a $4$-cycle in the graph $n$-TG$(R \setminus D)$. Thus, we conclude that $gr(n-TG(R \setminus D)) \leq 4$.
		\item Assume that $char(R) \neq 2$, $n$ is even, and $u^n = -1$ for some $u \in R$. Suppose $R \setminus D$ contains a cycle. Then there exist distinct $x, y \in R \setminus D$ such that $x \neq uy$ and $x^n + y^n \in D$. This implies that the sequence $x - y - (uy) - (ux) - x$ forms a $4$-cycle in the graph $n-TG(R \setminus D)$. Thus, the girth satisfies: $
		gr(n-TG(R \setminus D)) \leq 4$.
	\end{enumerate}
\end{proof}

\begin{example}
	\begin{enumerate}
		\item If $R=\mathbb{Z}_2 \times \mathbb{Z}_3$ and $D=\{0\} \times \mathbb{Z}_3 \  \cup \mathbb{Z}_2 \times \{0\}$. The $2-TG(D)$ has the cycle $(0,0)-(1,0)-(0,2)-(0,0)$  with $gr(n-TG(D))=3$.
		\item If $R=\mathbb{Z}_2 \times \mathbb{Z}_3$ and  $D=\{0\} \times \mathbb{Z}_3 \  \cup \mathbb{Z}_2 \times \{0\}$. The $2-TG(R \setminus D)$ is totally disconnected and $gr(n-TG(R \setminus D))=\infty$.
		\item If $R=\mathbb{Z}_2 \times \mathbb{Z}_3$, then the $3-TG(R \setminus D))$ has no cycle and therefore the $gr(n-TG(R\setminus D))=\infty$. If $R=\mathbb{Z}_3 \times \mathbb{Z}_3$, then the $3-TG(R \setminus D))$ has the cycle $(1,1)-(1,2)-(2,2)-(1,1)$. Therefore, the $gr(n-TG(R \setminus D))=3$.
	\end{enumerate}
\end{example}
We invite the reader to compare the following result with \cite[ Theorem 2.37]{1}.
\begin{theorem}
	For any integer $m \geq 2$, there exists an integral domain $R$ such that for every integer $n \geq 1$, the $n-TG(R)$ is connected and $diam(n-TG(R))=m$.
\end{theorem}

\begin{proof}
	Let $R = \mathbb{Z}_2[X_1,X2, ...,X_{m-1}]$ and let $ D = X_1R \cup X_2R \cup ...\cup X_{m-1}R \cup (1+X_1 + X_2 + ... + X_{m-1})R$. Take $F_i = X_1 + ... + X_i,$ so $F_{i+1} = F_i + X_{i+1}$ for $0 \leq i \leq m-1$. For some $d \in D$, we have $F_i^n+F_{i+1}^n=F_i^n+(F_i+X_{i+1})^n=F_i^n+F_i^n+dX_{i+1}+X_{i+1}^n=dX_{i+1}+X_{i+1}^n. $ Hence, $F_i^n+F_{i+1}^n \in D.$
\end{proof}

\begin{example}
	Let $R = \mathbb{Z}_2[X_1,X_2,X_3]$ and $D = X_1R \cup X_2R \cup X_3R \cup (1 + X_1 +X_2 + X_3)R$. Then, $0 - X_1 -(X_1 + X_2) - (X_1 + X_2 + X_3) - 1$ is the shortest path of length 4 between 0 and 1. For some $d \in D$, we have $(X_1 + X_2 + X_3)^n + 1^n
	= (X_1 + X_2 + X_3 + 1 + 1)^n + 1^n = (X_1 + X_2 + X_3 + 1)^n + d(X_1+X_2 + X_3 + 1)  + 1^n + 1^n=(X_1 + X_2 + X_3 + 1)^n + d(X_1+X_2 + X_3 + 1) \in D$.
\end{example}
{\bf Acknowledgment}

The first-named author would like to thank the second-named author for his patience and guidance during the work on this project. 
\bibliographystyle{amsalpha}

\end{document}